\documentclass[12pt]{amsart}
\usepackage{color,graphicx,array, amssymb, amscd,slashed,amsmath,amscd,latexsym,multirow,mathrsfs}
\usepackage[pdftex]{hyperref}
\usepackage{tikz}
\usepackage{tikz-cd}
\usepackage[hang,small]{caption}
\usepackage[subrefformat=parens]{subcaption}
\usetikzlibrary{fit,matrix}
\usepackage{float}
\usepackage{enumitem}
\newlist{steps}{enumerate}{1}
\setlist[steps, 1]{label = Step \arabic*:}

\newtheorem{theorem}{Theorem}[section]
\newtheorem{corollary}{Corollary}[section]
\newtheorem{lemma}{Lemma}[section]
\newtheorem{proposition}{Proposition}[section]
\newtheorem{remark}{Remark}[section]
\newtheorem{example}{Example}[section]
\newtheorem{definition}{Definition}[section]

\renewcommand{\labelenumi}{(\theenumi)}
\numberwithin{equation}{section}

\numberwithin{equation}{section}
\newcommand{\R}{\mathbb R}

\newcommand{\C}{\mathbb C}

\newcommand{\id}{\operatorname{id}}

\newcommand{\tr}{\operatorname{tr}}

\newcommand{\diag}{\operatorname{diag}}

\renewcommand{\Re}{\operatorname{Re}}
\renewcommand{\Im}{\operatorname{Im}}

\newcommand{\SU}{\mathrm {SU}(2)}

\numberwithin{equation}{section}

\usepackage{amsmath}	
\begin{document}
\title{The evolution of a curve induced by the Pohlmeyer-Lund-Regge equation}
\dedicatory{}
\author[S.-P.~Kobayashi]{Shimpei Kobayashi}
 \address{Department of Mathematics, Hokkaido University, Sapporo, 060-0810, Japan}
 \email{shimpei@math.sci.hokudai.ac.jp}
\author[Y.~Kogo]{Yuhei Kogo}
\address{Department of Mathematics, Hokkaido University, Sapporo, 060-0810, Japan}
\email{yh1123\_kg\_5813@eis.hokudai.ac.jp}
\author[N.~Matsuura]{Nozomu Matsuura}
\address{Department of Applied Mathematics, Fukuoka University, Fukuoka, 814-0180, Japan}
\email{nozomu@fukuoka-u.ac.jp} 
\thanks{The named first author is partially supported by Kakenhi 22K03304 and 
the second named author is partially supported by JST SPRING, Grant Number JPMJSP2119 and
 the third named author is partially supported by Kakenhi 23K03125.}
\subjclass[2020]{Primary~53A04, Secondary~37K10}
\keywords{Curve evolution; Pohlmeyer-Lund-Regge equation; solitons}

 \date{\today}
\pagestyle{plain}
\begin{abstract}
This paper investigates the evolution of space curves governed by the Pohlmeyer-Lund-Regge (PLR) equation, an integrable extension of the sine-Gordon equation.
We examine a specific type of curve evolution, known as the Lund-Regge evolution, and derive its representation in the Frenet frame.
We show the Frenet frame evolution aligns with the Lax system of the PLR equation and develop a construction method for curve families via the Sym formula. In conclusion, we describe the Lund-Regge evolution corresponding to the Date multi-soliton solutions to the PLR equation, with illustrations of curves and surfaces.
  \end{abstract}
\maketitle    

\section{Introduction}
The study of space curve evolution contributes to differential geometry, with applications extending to fluid dynamics and integrable systems.  One example is the \textit{vortex filament}, which has a notable impact on fluid dynamics \cite{Ricca}.
 Representing the vortex filament by $\gamma$,
 the associated system of partial differential equations
 governing its evolution is given by 
 $\dot{\gamma} = \gamma^{\prime \prime} \times \gamma^{\prime}$,
 where the derivatives with respect to the curve and time parameters are denoted by $'$ and $\cdot$, respectively, 
 and $\times$ denotes the cross product in $\mathbb{R}^3$.
 Using the Frenet frame, this equation can be reformulated with curvature $\kappa$ and the binormal vector $B$, resulting in $\dot{\gamma} = \kappa B$, also known as the \textit{binormal motion} of the filament \cite{JerrardSmets}. 
 
 Hasimoto \cite{Hashimoto} introduced a transformation, which 
  effectively reduced the nonlinear PDEs for curvature $\kappa$ and torsion $\tau$ of
 the vortex filament to the nonlinear Schrödinger equation, $i \dot q + q^{\prime \prime}+ 2|q|^2q = 0$,
 where $i = \sqrt{-1}$ is the imaginary unit, and the complex-valued function $q$ is defined
 by $q = \kappa \exp\left(i\int^s \tau ds\right)$ (the \textit{Hasimoto transformation}). 
 This nonlinear Schr\"odinger equation was widely recognized as an integrable system, and both its soliton and quasi-periodic solutions have been extensively studied \cite{Belokolos}. Calini and Ivey \cite{CaliniIvey} have rigorously explored the differential geometric properties of vortex filaments.
   
 Parallel to the study of vortex filaments, 
 another significant equation arises in the 
 context of space curve evolution, namely, the sine-Gordon equation \cite{Lamb}: This equation, well-established in the study of constant negative curvature surfaces in three-dimensional 
 Euclidean space \cite{Mclachlan,Sym}  and recognized in quantum field theory as the sine-Gordon model  \cite{Lund}, constitutes a foundational integrable system.
 In 1978, Lund and Regge \cite{LundRegge} discovered a generalization of the sine-Gordon equation 
 and formulated it via the time evolution of a space curve $\gamma$ as the following system of nonlinear partial differential equations
\[
 \dot{\gamma}^{\prime} = \gamma^{\prime} \times \dot{\gamma}.
\]
We henceforth refer to this equation as the \textit{Lund-Regge evolution}. When a curve $\gamma$ moves according to the Lund-Regge evolution, its curvature and torsion satisfy a special system of partial differential equations that extends the sine-Gordon equation.
This system, also derived by Pohlmeyer \cite{Pohlmeyer} from a viewpoint of integrable systems theory, is now known as the \textit{Pohlmeyer-Lund-Regge {\normalfont (}PLR {\normalfont for short)} equation}. In 1982, Date \cite{Date} applied an extension of Krichever's method to construct soliton and quasi-periodic solutions to the PLR equation, 
introducing an approach to finding solutions by directly solving a special linear equation through Cramer’s rule, a technique subsequently referred to as the \textit{Date direct method} (see Section \ref{sbsc:Date}  for further details). Figure \ref{fig:2-sliton} describes a surface 
induced by a $2$-soliton solution of the PLR equation.

This paper establishes the framework for space curve evolution using the Frenet frame, alongside the partial differential equations that describe the evolution of curvature and torsion (see Section \ref{sbsc:spacecurve}). We further reformulate these representations using the $2 \times 2$ matrix model through the identification $\mathbb{R}^3 \cong \mathfrak{su}(2)$. Additionally, we summarize the PLR equation and its associated Lax system, with a particular emphasis on the Date construction of $N$-soliton solutions 
(Sections \ref{sbsc:PLR} and \ref{sbsc:Date}).

 \begin{figure}[tb]
        \centering
        \includegraphics[keepaspectratio]{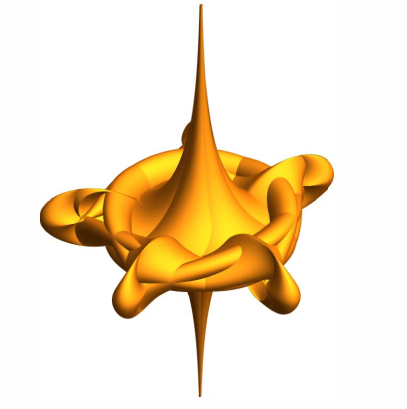}
        \caption{A surface
        induced by a $2$-soliton solution of the Pohlmeyer-Lund-Regge equation, see Section \ref{sbsc:Nsolitoncurves} for details.}
        \label{fig:2-sliton}
\end{figure}

The core of this paper is developed in Section \ref{sc:evoPLR}. In that section, we focus on the $2 \times 2$ matrix-valued Frenet frame associated with the Lund-Regge evolution, and derive the system of partial differential equations induced by its time evolution (Theorem \ref{thm:PLRevolution}). We also define the complex-valued function $q = \kappa \exp \left( i \int^s (\tau - 1) ds \right)$ as an analogue to the Hasimoto transformation. By utilizing an appropriate gauge transformation, we demonstrate that the time evolution of the Frenet frame aligns precisely with the Lax system of the PLR equation, with the spectral parameter set to 1 (see Corollary \ref{coro:PLRevolution}). This correspondence reveals that the function $q$ satisfies the following partial differential equation:
\[
\dot q^{\prime} + \frac{1}{2} q \int^s \left( |q|^2 \right)^{\boldsymbol{\cdot}} ds = 0,
\]
 an equation analogous to the nonlinear Schr\"odinger equation, and we will refer to this equation as also the \textit{PLR equation}.
  Moreover, we show that by taking the logarithmic derivative of the Lax equation solutions with respect to the spectral parameter, a family of space curves meeting the Lund-Regge evolution can be constructed (Theorem \ref{thm:Representation}, also known as the  \textit{Sym formula}).
When the torsion is identically $\tau = 1$, the dependent variable $q$ in the PLR equation degenerates to a real-valued function, and in this case the PLR equation itself is reduced to the sine-Gordon equation (Remark \ref{rem} \ref{rem:sG}).
Additionally, we will describe the Lund-Regge evolution corresponding to the Date $N$-soliton solutions, while presenting illustrations of curves and surfaces as visual aids to support our explanation (Proposition \ref{prp:Nsolitoncurve} and Figures \ref{fig:1-3swptsurfaces}, \ref{fig:4PLRsoliton}, \ref{fig:4sGsoliton}).

\section{Preliminaries}\label{sc:Pre}
 In this section, we will first collect basic results on a curve evolution in the Euclidean three-space.
 We will then recall the Pohlmeyer-Lund-Regge equation in terms of the Lax pair 
 and its $N$-soliton solutions \cite{Date}.
 \subsection{Evolutions of space curves}\label{sbsc:spacecurve}
 For  a given regular space curve $\gamma$ in the Euclidean three-space, let us consider the
 Frenet frame
\[
T = \frac{\gamma^{\prime}}{|\gamma^{\prime}|},\quad
N =  B \times T, \quad 
B = \frac{\gamma^{\prime} \times \gamma^{\prime \prime}}{|\gamma^{\prime} \times \gamma^{\prime \prime}|},
\]
where $\prime$ denotes the derivative with respect to the parameter of the curve $\gamma$ and 
$\times$ denotes the Euclidean cross-product. Then it is well known \cite{doCarmo}
that the  curvature $\kappa$ and the torsion $\tau$ of $\gamma$ can be represented by 
\begin{equation}\label{eq:kappatau}
\kappa = \frac{|\gamma^{\prime} \times \gamma^{\prime \prime}|}{|\gamma^{\prime}|^3}, 
\quad
\tau = 
\frac{\left\langle\gamma^{\prime}\times  \gamma^{\prime \prime},
\gamma^{\prime \prime \prime}\right\rangle
}
{|\gamma^{\prime} \times \gamma^{\prime \prime}|^2},
\end{equation}
where $\langle \cdot,\cdot \rangle$ is the inner product in $\R^3$.
Note that the Frenet frame field $F = (T, N, B)$ takes values in
\[\operatorname{SO}(3, \R)
=\{A \in M(3, \R) \mid A A^{\mathsf{T}} = \id,\, \det A =1\}. 
\]
Then by the Frenet-Serret equation, $F$ satisfies 
\begin{equation}\label{eq:Frenet-serre}
F^{\prime} =FL, \quad L = \begin{pmatrix} 
0 & -|\gamma^{\prime}| \kappa  &  0 \\
|\gamma^{\prime}| \kappa  &  0 & - |\gamma^{\prime}| \tau \\
0 & |\gamma^{\prime}| \tau & 0
\end{pmatrix}.
\end{equation}
Let us consider an evolution of the curve $\gamma$ and its derivative will be denoted by the dot, and 
assume $\gamma$ is regular for the evolution\footnote{
Strictly speaking, we should use a different symbol for the evolution of $\gamma$ with a deformation parameter
$|t| \leq \epsilon$ with sufficiently small $\epsilon >0$, that is, 
 $\tilde \gamma$ such that $\tilde \gamma|_{t=0} = \gamma$, however, we denote the evolution of 
 $\gamma$ by the same symbol.
 }, and  the Frenet frame is defined along $\gamma $ as $F = (T, N, B)$ and it 
satisfies \eqref{eq:Frenet-serre}. 
We denote the infinitesimal deformation of $\gamma$ by 
\begin{equation}\label{eq:gammadot}
\dot \gamma = a T + b N + c B,    
\end{equation}
 where $a, b, c$ are functions of two variables with respect to the curve parameter and the deformation parameter.
Then it can be verified that the evolution of the Frenet frame $F$ can be represented as
\begin{equation}\label{eq:Fevo}
    \dot F = F M, \quad M = \begin{pmatrix}
 0 & *&* \\        
 \frac{b^{\prime}}{|\gamma^{\prime}|} + a \kappa -c \tau & 0& *\\        
 \frac{c^{\prime}}{|\gamma^{\prime}|} + b \tau   & 
 \left(\frac{c^{\prime}}{|\gamma^{\prime}|} + b \tau \right)^{\prime} \frac{1}{|\gamma^{\prime}| \kappa} +
 \left(\frac{b^{\prime}}{|\gamma^{\prime}|} + a \kappa - c \tau  \right) \frac{\tau}{\kappa}&0 
    \end{pmatrix},
\end{equation}
 see Appendix \ref{sbsc:evo} for details.
 Here we omit the upper right entries of $M$ and they are determined by the condition $M$
 taking values in 
 \[
 \mathfrak {so}(3, \R) = \{A \in M(3, \R) \mid A + A^{\mathsf{T}} = 0 \}.
 \]
 The compatibility condition $(\dot F)^{\prime} = {(F^{\prime})}^{\boldsymbol{\cdot}}$ implies the following proposition.
\begin{proposition}\label{prp:curve}
Let $\gamma$ be an evolution of a space curve given in \eqref{eq:gammadot}
and denote the length element of 
$\gamma$ by 
\[
\ell = |\gamma^{\prime}| >0.
\]
Then the following system of partial differential equations holds:
\begin{align}
\label{eq:compatibility1}
\dot \ell  &= a^{\prime} - b \ell \kappa, 
\\
\label{eq:compatibility2}
(\ell \kappa)^{\boldsymbol{\cdot}} &= \left(\frac{b^{\prime}}{\ell} + a \kappa -c \tau \right)^{\prime}
- \left(\frac{c^{\prime}}{\ell} + b\tau \right) \ell \tau, \\   
\label{eq:compatibility3}
(\ell \tau)^{\boldsymbol{\cdot}} &= 
\left(\frac{c^{\prime}}{\ell} + b\tau \right) \ell \kappa +
\left\{
\left(\frac{c^{\prime}}{\ell} + b \tau \right)^{\prime} \frac1{\ell \kappa}
+\left(\frac{b^{\prime}}{\ell} + a \kappa - c \tau \right) \frac{\tau}{\kappa}
\right\}^{\prime}.
\end{align}    
 Conversely, let $a, b, c, \kappa, \tau, \ell$ be functions of two variables such that 
 the equations \eqref{eq:compatibility1}, \eqref{eq:compatibility2} and \eqref{eq:compatibility3}
 hold. Then there exists the evolution of a space curve $\gamma$ such that the length element, 
 the curvature, and the torsion of 
 $\gamma$ are given by  $\ell = |\gamma^{\prime}|$ and $\kappa$ and $\tau$, respectively. 
 Furthermore, the deformation of $\gamma$ is represented as in \eqref{eq:gammadot}.
\end{proposition}
\begin{proof}
  First, we show the equation \eqref{eq:compatibility1}. From \eqref{eq:Frenet-serre} and \eqref{eq:gammadot}, the following equation holds:
  \begin{equation}\label{eq:gammadotprime}
  \dot \gamma^{\prime} = (a^{\prime} - b |\gamma^{\prime}| \kappa) T + (b^{\prime} + a |\gamma^{\prime}| \kappa - c |\gamma^{\prime}| \tau) N + (c^{\prime} + b |\gamma^{\prime}| \tau) B.
  \end{equation}
  Thus $\langle \dot \gamma^{\prime}, \gamma^{\prime} \rangle = |\gamma^{\prime}| (a^{\prime} - b |\gamma^{\prime}| \kappa)$ holds, and we have
$|\gamma^{\prime}|^{\boldsymbol{\cdot}} = a^{\prime} - b |\gamma^{\prime}| \kappa$, which is \eqref{eq:compatibility1} since we set $|\gamma^{\prime}| = \ell$. Next, we will calculate the compatibility condition of $L$ and $M$ in \eqref{eq:Frenet-serre} and \eqref{eq:Fevo}. The compatibility condition $[L,M] + M^{\prime} - \dot L = 0$ is
  \begin{equation}\label{eq:so3generalLMcomp}
       \begin{pmatrix}
        0 & *&* \\        
        -l_{32} m_{31} + m_{21}^{\prime} - \dot l_{21} & 0& *\\        
        l_{32} m_{21} - l_{21} m_{32} + m_{31}^{\prime}   & 
        l_{21} m_{31} + m_{32}^{\prime} - \dot l_{32}&0 
        \end{pmatrix}
        = 0,
  \end{equation}
  where $l_{ij}, m_{ij} (i, j \in \{1, 2, 3\})$ denote the $(i, j)$-entries of $L$ and $M$, respectively. Since $L$ and $M$ take values in $\mathfrak{so}(3,\R)$, the upper right entries of the left-hand side of \eqref{eq:so3generalLMcomp} are determined by the lower left ones. We denote these by $\ast$ as they need not be considered. A straightforward computation shows that the $(3,1)$-entry of the left-hand side is zero.
Therefore \eqref{eq:so3generalLMcomp} reduces to the pair of equations
  \begin{gather*}
   - |\gamma^{\prime}| \tau \left(\frac{c^{\prime}}{|\gamma^{\prime}|} + b \tau\right) + \left(\frac{b^{\prime}}{|\gamma^{\prime}|} + a \kappa - c \tau\right)^{\prime} - (|\gamma^{\prime}| \kappa)^{\boldsymbol{\cdot}} = 0,\\
     |\gamma^{\prime}| \kappa \left(\frac{c^{\prime}}{|\gamma^{\prime}|} + b \tau\right) + \left(\left(\frac{c^{\prime}}{|\gamma^{\prime}|} + b \tau \right)^{\prime} \frac{1}{|\gamma^{\prime}| \kappa} +
    \left(\frac{b^{\prime}}{|\gamma^{\prime}|} + a \kappa - c \tau  \right) \frac{\tau}{\kappa}\right)^{\prime} - (|\gamma^{\prime}| \tau)^{\boldsymbol{\cdot}} = 0.
  \end{gather*}
  These equations are \eqref{eq:compatibility2} and \eqref{eq:compatibility3}, respectively. Conversely, let $\Tilde{L}$ and $\Tilde{M}$ be $4 \times 4$ matrices
  \begin{equation}
      \Tilde{L} = 
        \begin{pmatrix}
            &&&\ell \\
            &L&&0\\
            &&&0\\ 
            0&0&0&0
        \end{pmatrix},\quad
        \Tilde{M} = 
        \begin{pmatrix}
            &&&a\\
            &M&&b\\
            &&&c\\ 
            0&0&0&0
        \end{pmatrix},
  \end{equation}
  where $3\times 3$ matrices $L$ and $M$ are defined in \eqref{eq:Frenet-serre} and \eqref{eq:Fevo} with $|\gamma^{\prime}| = \ell $, respectively. Note that $a, b, c$ are arbitrary functions of two variables.
  Then it can be verified that the compatibility conditions of $\Tilde{L}$ and $\Tilde{M}$ are \eqref{eq:compatibility1}, \eqref{eq:compatibility2}, and \eqref{eq:compatibility3}. Therefore there exists a $4\times 4$ matrix-valued function $\tilde{F} = \tilde{F}(s,t)$ that is a solution to the partial differential equations:
  \begin{equation}
      \Tilde{F}^{\prime} = \Tilde{F}\Tilde{L},\quad  \dot{\Tilde{F}} = \Tilde{F}\Tilde{M},
  \end{equation}
  and by choosing $\Tilde{F}(0,0) = \textrm{id}_4$, we have
  \begin{equation}
      \Tilde{F} = 
      \begin{pmatrix}
          &&&\Tilde{F}_{14}\\
            &F &&\Tilde{F}_{24}\\
            &&&\Tilde{F}_{34}\\ 
            0&0&0&1
      \end{pmatrix}.
  \end{equation}
  Then $\gamma(s,t) := (\Tilde{F}_{14}(s,t),\Tilde{F}_{24}(s,t),\Tilde{F}_{34}(s,t))^{\mathsf{T}}$ 
  gives the evolution of a curve by \eqref{eq:gammadot}.
\end{proof}
We now identify $\mathfrak{so}(3, \R)$ with 
 \[
 \mathfrak{su}(2) =
  \{ X \in \mathfrak{sl}(2, \mathbb C) \mid  X + \overline{X}^{\mathsf{T}} = 0 \}
  \]
 by the following explicit map:
\begin{equation}\label{eq:identification}
\begin{pmatrix} 
0 & -p & -q \\
p & 0 & -r \\
q & r & 0 
\end{pmatrix} 
\in \mathfrak{so}(3, \mathbb R) \longleftrightarrow 
\frac{1}{2}
\begin{pmatrix} 
i r &  -p - i q \\
p- i q &  -ir 
\end{pmatrix}  
\in \mathfrak{su}(2).
\end{equation}
 Note that, under the above identification, the evolution of the Frenet frame $F$ taking values in $\operatorname{SO}(3,\R)$
 can be identified with a map $\mathcal F$ taking values in $\operatorname{SU}(2)$, see 
 for an explicit correspondence \cite{Inoguchi}.
 We now arrive at the following fundamental theorem of an evolution in $\mathfrak{su}(2)  \cong \R^3$.
\begin{theorem}
 Let $\mathcal F$ be a map taking values in $\operatorname{SU}(2)$ 
 by the correspondence in \eqref{eq:identification} to the evolution of the Frenet frame $F$ 
 given by \eqref{eq:Frenet-serre} and \eqref{eq:Fevo}. 
 Then the following system of partial differential equations holds:
\begin{equation}\label{eq:SU2Frame}
\mathcal F^{\prime} = \mathcal F\mathcal  L, \quad 
\dot {\mathcal F} = \mathcal F \mathcal  M,
\end{equation}
with 
\begin{align}\label{eq:LM}
\mathcal L = 
\frac{ |\gamma^{\prime}| }2
\begin{pmatrix} 
i \tau &  -\kappa  \\
\kappa   & -i \tau 
\end{pmatrix}, \quad
\mathcal M
= 
\frac12
\begin{pmatrix} 
\frac{i}{\kappa} \left(\frac{m_{31}^{\prime}}{|\gamma^{\prime}|} + m_{21} \tau \right) & -m_{21} - i m_{31}  \\
m_{21} - i m_{31}   & 
-\frac{i}{\kappa} \left(\frac{m_{31}^{\prime}}{|\gamma^{\prime}|} + m_{21} \tau \right) 
\end{pmatrix},
\end{align}
 where $m_{ij} (i, j \in \{1, 2, 3\})$ denotes the $(i, j)$-entry of $M$ in \eqref{eq:Fevo}
 which are explicitly given by 
\begin{equation}\label{eq:explicitm}
 m_{21} = \frac{b^{\prime}}{|\gamma^{\prime}|} + a \kappa-c \tau, \quad    
 m_{31} = \frac{c^{\prime}}{|\gamma^{\prime}|} + b \tau. 
\end{equation}
 The compatibility equations of the system \eqref{eq:SU2Frame} are \eqref{eq:compatibility2} and \eqref{eq:compatibility3}.
\end{theorem}
\begin{proof}
  The compatibility condition $[\mathcal L, \mathcal M] + \mathcal M ^{\prime} -\dot {\mathcal L} = 0$ is
  \begin{equation}\label{eq:su2generalLMcomp}
        \begin{pmatrix}
          \mathcal L_{12} \mathcal M_{21} - \mathcal M_{12} \mathcal L_{21} + \mathcal M_{11}^{\prime} - (\mathcal L_{11})^{\boldsymbol{\cdot}} & * \\
          \mathcal L_{21} \mathcal M_{11} + \mathcal L_{22} \mathcal M_{21} - \mathcal M_{21} \mathcal L_{11} - \mathcal M_{22} \mathcal L_{21} + \mathcal M_{21}^{\prime} - (\mathcal L_{21})^{\boldsymbol{\cdot}} & *
        \end{pmatrix} = 0,
  \end{equation}
  where $\mathcal L_{ij}, \mathcal M_{ij} (i, j \in \{1, 2\})$ denote the $(i, j)$-entries of $\mathcal L$ and $\mathcal M$ in \eqref{eq:LM}. The second column of the left-hand side of \eqref{eq:su2generalLMcomp} is determined by the first column, so we have omitted it. A straightforward computation shows that the $(1,1)$-entry and the $(2,1)$-entry of the left-hand side in \eqref{eq:su2generalLMcomp} are equivalent to \eqref{eq:compatibility2} and \eqref{eq:compatibility3}, respectively. 
\end{proof}

\subsection{The Pohlmeyer-Lund-Regge equation}\label{sbsc:PLR}
 We now recall the Pohlmeyer-Lund-Regge equation \cite{LundRegge, Pohlmeyer} according to the formulation in
 \cite{Date}.
 Let $\Psi$ be a $2\times 2$ matrix-valued map satisfying the following 
 system of PDEs, the so-called \textit{Lax pair}:
\begin{align}\label{eq:DateLaxpair}
\Psi^{-1} \Psi^{\prime} =  
\frac{1}{2} \begin{pmatrix} i \lambda &  a \\ -\bar a & -i \lambda \end{pmatrix},
 \;\;
\Psi^{-1} \displaystyle   \dot \Psi = 
\frac{i}{2 \lambda }  \begin{pmatrix} \cos u &  -\exp (i \omega) \sin u\\
  -\exp (- i \omega) \sin u & -\cos u,
\end{pmatrix},
\end{align}
 where $\lambda$ is a constant in $\C^{\times} = \C \setminus \{0\}$, the \textit{spectral parameter}, 
 and $u$ is a real-valued function of two variables\footnote{The Lax pair has been transposed and the definition of $a$ is multiplied by $i/2$ in 
 \cite{Date}.}. 
 Moreover, a complex-valued function $a$ and a real-valued function $\omega$ of two variables 
 are given by 
 $u$ and another real-valued function $v$ as
\begin{equation}\label{eq:PLRaomega}
 a = -\frac{\{\exp(i \omega) \sin u\}^{\prime}}{\cos u} 
,\quad \omega^{\prime} = \frac{v^{\prime}  \cos u}{2 \cos^2 (u/2)},\quad
 \dot \omega = \frac{\dot v }{\cos^2 (u/2)},
 \end{equation}
 where we assume 
 $0<u<\pi/2$.
 Note that $\omega$ is defined as the solution to the pair of PDEs in \eqref{eq:PLRaomega}.
 Then, the compatibility conditions of the above Lax pair are
 the so-called \textit{Pohlmeyer-Lund-Regge equation}
:
\begin{equation}\label{eq:PLR}
\dot u^{\prime} - \frac{v^{\prime} \dot v \sin (u/2)}{2 \cos^3 (u/2)} + \sin u  = 0, \quad
\dot v^{\prime} +\frac{u^{\prime} \dot v  + \dot u v^{\prime} }{\sin u}  = 0.
\end{equation}
If the function $v$ 
is constant, then \eqref{eq:PLR} reduces to the sine-Gordon equation:
 \begin{equation}\label{eq:sine-Gordon}
     \dot u^{\prime} + \sin u  = 0.
 \end{equation}
 The matrix-valued function $\Psi$ in \eqref{eq:DateLaxpair}
will be called the \textit{wave function}.

\subsection{The Date \texorpdfstring{$N$-}-solitons}\label{sbsc:Date}
We recall the method of Date in \cite[Section 3]{Date} for constructing $N$-soliton solutions of Pohlmeyer-Lund-Regge equation \eqref{eq:PLR} and the wave function in \eqref{eq:DateLaxpair}. Let us assume that the wave function
\[
 \Psi (s, t, \lambda) =
\begin{pmatrix}
\Psi_{11}(s,t,\lambda) & \Psi_{12}(s,t,\lambda)\\
\Psi_{21}(s,t,\lambda) & \Psi_{22}(s,t,\lambda)
\end{pmatrix}
\] 
 has the following form:
\begin{align*}
    \Psi_{11}(s,t,\lambda)&=\left(\lambda^N+\sum_{j=0}^{N-1}\nolimits\psi_{1j}(s,t)\lambda^j\right)\exp\left\{\frac{i}{2}(\lambda s +\lambda^{-1}t)\right\},\\
    \Psi_{21}(s,t,\lambda)&=-\left(\sum_{j=0}^{N-1}\nolimits\psi_{2j}(s,t)\lambda^j\right)\exp\left\{-\frac{i}{2}(\lambda s +\lambda^{-1}t)\right\},\\
    \Psi_{12}(s,t,\lambda)&=-\overline{\Psi_{12}(s,t,\bar\lambda)},\quad \Psi_{22}(s,t,\lambda)=\overline{\Psi_{11}(s,t,\bar\lambda)},
\end{align*}
 where $\psi_{ij}(s,t) \left(i \in \{1, 2\}, j \in \{0, \dots, N-1\}\right)$ are $\lambda$-independent complex-valued functions. 
For each number $j \in \{1, \dots, N\}$, let $\alpha_j$ be mutually distinct complex numbers such that their imaginary parts $\Im \alpha_j$ have the same sign for all $j$, and let $c_j$ be arbitrary complex numbers. Then, we consider the solution $\psi = (\psi_{10},\cdots,\psi_{1,N-1},\psi_{20},\cdots,\psi_{2,N-1})^{\mathsf{T}} \in \mathbb{C}^{2N}$ of the following linear equation:
\begin{equation}\label{eq:linearequation}
T_0\psi = b.
\end{equation}
 Here the vector $b \in \mathbb{C}^{2N}$ is defined by
\begin{align*}
b =-\left((\alpha_1)^Ne(\alpha_1),\dots,(\alpha_N)^Ne(\alpha_N), 
 (\overline{\alpha_1})^N \overline{c_1} e(\overline{\alpha_1}),\dots, (\overline{\alpha_N})^N \overline{c_N} e(\overline{\alpha_N})\right)^{\mathsf{T}}
\end{align*}
 with $e(\lambda) =  \exp(\frac{i}2( \lambda s + \lambda^{-1} t ))$, and the $2N \times 2N$ matrix $T_0$ is defined by 
\begin{align*}
T_0 =
\begin{pmatrix}
    EA & -CE^{-1}A \\
    \overline{CE^{-1}A} & \overline{EA}
\end{pmatrix}
\end{align*}
with $N\times N$ matrices
\[
A = \begin{pmatrix}
1 & (\alpha_1)^1 & \cdots & (\alpha_{1})^{N-1}\\
\vdots  & \vdots & \ddots & \vdots \\
1 & (\alpha_N)^1 & \cdots & (\alpha_{N})^{N-1}
\end{pmatrix}, 
\quad 
E = 
\begin{pmatrix}
e(\alpha_1)&&O\\
 &\ddots &\\
 O&&e(\alpha_N)
\end{pmatrix}, \quad 
C = 
\begin{pmatrix}
c_1 &&O\\
 &\ddots &\\
  O&&c_N
\end{pmatrix}.
\]
Note that under the conditions on $\alpha_j$, the coefficient matrix $T_0$ in \eqref{eq:linearequation} is non-singular.
In fact, by Cramer's rule, the functions $\psi_{ij}$ are explicitly computed as follows:
\begin{equation}\label{eq:explicitpsi}
    \psi_{1j}(s,t) = \frac{\det T_{j+1}}{\det T_0} = \frac{d_{j+1}}{d_0}, \quad
    \psi_{2j}(s,t) = \frac{\det T_{N+j+1}}{\det T_0} = \frac{d_{N+j+1}}{d_0}
\end{equation}
for $j \in \{0,\dots,N-1\}$. Here $T_k$ denotes the matrix
formed by replacing the $k$-th column of $T_0$ by the vector $b$ in \eqref{eq:linearequation}, and $d_j$ are defined for $j \in \{0, \dots, N\}$ as
\[
d_j = \det T_j.
\]
Following Date's discussion in \cite[section 3]{Date}, the parameters $\alpha_1,\dots,\alpha_N,c_1,\dots,c_N$ determine a unique wave function $\Psi(s,t,\lambda)$ that satisfies the conditions
\[
(1,c_j)\Psi(s,t,\alpha_j)=0
\]
for all $j \in \{1,\dots,N\}$.
Furthermore, we can show that the function $\Psi(s,t,\lambda)$ satisfies the linear differential equations
\begin{equation}\label{eq:NsolitonLaxpair}
\left\{ \,
\begin{aligned}
    \Psi^{-1}\Psi^{\prime} &= \frac12
    \begin{pmatrix}
        i\lambda & i \overline{d_{2N}/d_0} \\
        id_{2N}/d_0 & -i\lambda
    \end{pmatrix}, \\
    \Psi^{-1}\dot\Psi &= \frac{i}{2\lambda(|d_1|^2+|d_{N+1}|^2)}
    \begin{pmatrix}
        |d_1|^2-|d_{N+1}|^2 & 2\overline{d_1d_{N+1}} \\
        2d_1d_{N+1} & -|d_1|^2+|d_{N+1}|^2
    \end{pmatrix}.
\end{aligned}
\right.
\end{equation}
Comparing \eqref{eq:NsolitonLaxpair} with \eqref{eq:DateLaxpair}, we have 
that a solution of \eqref{eq:PLR} can be given as follows:
\begin{equation}\label{eq:Nsolitonsolution}
  a=i\,\overline{\left(\frac{d_{2N}}{d_0}\right)},\quad
  u=\arccos\left(\frac{|d_1|^2-|d_{N+1}|^2}{|d_1|^2+|d_{N+1}|^2}\right),\quad
  v=2\arg\overline{\left(\frac{d_{N+1}}{d_0}\right)}+v_0,
\end{equation}
where $v_0 \in \R$,
and moreover $\Psi(s,t,\lambda)$ satisfies \eqref{eq:DateLaxpair}.

For example, to give the $N$-soliton solutions to the sine-Gordon equation, we choose parameters $\alpha_1,\dots, \alpha_N$ and $c_1,\dots, c_N$ such that there exists a suitable permutation $\sigma\in \mathfrak{S}_N$ of $\{1,\dots,N\}$ satisfying the conditions
\begin{equation}\label{eq:sGcondition}
    \overline{\alpha_j} = -\alpha_{\sigma(j)},\quad
    \overline{c_j} = -c_{\sigma(j)}
\end{equation}
for all indices $j$. Then we have for $k \in \{1,\dots,2N\}$ that
\begin{equation}
    \frac{d_{k}}{d_0} = \overline{\left(\frac{d_k}{d_0}\right)}.
\end{equation}
Now \textit{v} is constant by the third equation in \eqref{eq:Nsolitonsolution}, and thus \textit{u} is a solution of the sine-Gordon equation. 

\section{Special curve evolutions and the Pohlmeyer-Lund-Regge equation}\label{sc:evoPLR}
 In this section, we will consider a special curve evolution $\gamma$ in $\R^3$ according to the Pohlmeyer-Lund-Regge 
 equation, the so-called \textit{Lund-Regge} evolution \cite{LundRegge}, see 
 Definition \ref{def:LR}.
 We will show that the Lund-Regge evolution is equivalent to 
 the pair of nonlinear PDEs, see Theorem \ref{thm:PLRevolution}. 
 Next we will rephrase the Pohlmeyer-Lund-Regge equation in \eqref{eq:PLR}, the pair of PDEs,
 in terms of a single PDE 
 with respect to a complex-valued function in Corollaries \ref{coro:PLRevolution}  and \ref{coro:Laxpair}, 
 which is an analogue of the \textit{Hasimoto transformation} $q = \kappa \exp (\int^s \tau ds )$
 of the nonlinear Schr\"odinger equation \cite{Hashimoto}.
 Then the representation formula of the Lund-Regge evolution in terms of the logarithmic derivative of 
 the Lax pair with respect to the spectral parameter, the so-called \textit{Sym-formula} will be given, 
 Theorem \ref{thm:Representation}. Finally, we will give examples of $N$-soliton curves. 
 
\subsection{The Lund-Regge evolution}\label{sbsc:LRevo}
Let us start with the following definition:
\begin{definition}\label{def:LR}
 A space curve evolution $\gamma$ will be called the {\rm Lund-Regge evolution} if 
 the following equation holds:
\begin{equation}\label{eq:Lund-Regge}
\dot \gamma^{\prime} =  \gamma^{\prime} \times \dot \gamma.
\end{equation}
 Furthermore, the Lund-Regge evolution will 
 be called {\rm regular} if $\gamma^{\prime} \times \dot \gamma \neq 0$ holds.
 \end{definition}
\begin{remark}
 The original definition of the Lund-Regge evolution in 
 \cite[(3.2), (3.4) and (3.5)]{LundRegge} requires additional conditions:
\[
|\gamma^{\prime}| = \ell, \quad |\dot \gamma| = \ell^{-1}
\]  
for some positive constant $\ell>0$. However, the following lemma implies that we do not need the constant 
length conditions.
 \end{remark}
\begin{lemma}
 For a regular Lund-Regge evolution $\gamma$, by
 appropriate change of coordinates, 
 \[
 |\gamma^{\prime}| = |\dot \gamma|^{-1} = \ell
 \]
 holds.
\end{lemma}
\begin{proof}
By the equation of Lund-Regge evolution \eqref{eq:Lund-Regge}, we have that
\[
|\gamma^{\prime}|^{\boldsymbol{\cdot}}
= |\dot \gamma|^{\prime} =0,
\]
thus $|\gamma^{\prime}|$ and $|\dot \gamma|$
 depend only on the curve parameter and 
 the evolution parameter, respectively.
 Therefore, the change of coordinates implies the claim.
\end{proof}

\begin{proposition}
Let $\gamma$ be a Lund-Regge evolution  such that $|\gamma^{\prime}| =1$.
Expressing the direction of the evolution as
$\dot{\gamma} = aT+bN+cB$ via the Frenet frame, the defining equation \eqref{eq:Lund-Regge} is equivalent to the following equations:
\begin{equation}\label{eq:LRcond2}
 a^{\prime} - b \kappa  =0, \quad     
 b^{\prime} + a \kappa  - c \tau = -c, \quad     
 c^{\prime} + b \tau = b.
\end{equation}
 Moreover, the matrices $\mathcal L$ and $\mathcal M$ in \eqref{eq:SU2Frame}
are simplified as
\begin{align}\label{eq:LM2}
\mathcal L = 
\frac{1}{2}
\begin{pmatrix} 
i \tau& -\kappa\\
\kappa   &-i \tau
\end{pmatrix}, \quad 
\mathcal M
= 
\frac{1} {2 }
\begin{pmatrix} 
\frac{i}{\kappa} \left( b^{\prime} - c \tau \right)  &  c - i b   \\
-c - i b   & -\frac{i}{\kappa} \left( b^{\prime}   - c \tau \right)  
\end{pmatrix}
\end{align}
with the compatibility condition
\begin{equation}\label{eq:LRcomp}
\dot  \kappa = - b, \quad
 \dot \tau  =- \left( \frac{c}{\kappa}\right)^{\prime}.
\end{equation}
\end{proposition}
\begin{proof}
Recalling the general computational result \eqref{eq:gammadotprime} on the evolution of a space curve, we see that the defining equation of the Lund-Regge evolution \eqref{eq:Lund-Regge} is equivalent to the equation
  \begin{align}\label{eq:LR}
       (a^{\prime} - b |\gamma^{\prime}| \kappa) T + (b^{\prime} + a |\gamma^{\prime}| \kappa - c |\gamma^{\prime}| (\tau - 1)) N + (c^{\prime} + b |\gamma^{\prime}| (\tau - 1)) B = 0.
  \end{align}
  Therefore, the equations \eqref{eq:LRcond2} describe the Lund-Regge evolution \eqref{eq:Lund-Regge} with $|\gamma^{\prime}| = 1$. By using \eqref{eq:LRcond2}, the equations in \eqref{eq:explicitm} are simplified as
  \begin{equation}
    m_{21} = -c, \quad    
    m_{31} = b.
   \end{equation}
   Therefore, the matrices $\mathcal L$ and $\mathcal M$ in the evolution of the Frenet frame $\mathcal F$ in \eqref{eq:SU2Frame} can be simplified to \eqref{eq:LM2}. The compatibility condition between \eqref{eq:LM2} is \eqref{eq:LRcomp}.
\end{proof}
 From \eqref{eq:LRcomp} and the first equation in \eqref{eq:LRcond2}, the coefficient functions $a, b, c$ of the Lund-Regge evolution $\gamma$ can be represented by its curvature $\kappa$ and torsion $\tau$ as
 \begin{equation}\label{eq:abcPLR}
 a = -\int^s \dot \kappa \kappa \, ds 
\quad 
 b = -  \dot \kappa, \quad 
 c = - \kappa \int^s \dot \tau \, ds.
 \end{equation}
The second and third equations in \eqref{eq:LRcond2} respectively 
becomes  
\begin{gather}\label{eq:kappatauPLR1}
\frac{\dot \kappa ^{\prime}}{\kappa} - (\tau - 1)\int^s \dot \tau ds
+ \int^s \kappa \dot \kappa ds = 0,\\
\left(\kappa \int^s \dot \tau ds\right)^{\prime} + \dot \kappa (\tau - 1) = 0.
\label{eq:kappatauPLR2}
\end{gather}
We now summarize the discussion to this point as the following theorem.
\begin{theorem}\label{thm:PLRevolution}
Regarding the Lund-Regge evolution, we have the following$:$
\begin{enumerate}
\item[{\rm (1)}]
Let $\gamma$ be the Lund-Regge evolution such that $|\gamma'|=1$. Then it satisfies \eqref{eq:gammadot} with the coefficient functions \eqref{eq:abcPLR}, and its Frenet frame $\mathcal{F}$ evolves according to \eqref{eq:SU2Frame}
with the matrices
\begin{align}\label{eq:LMrealfinal}
\mathcal L = 
\frac{1} 2
\begin{pmatrix} 
i \tau& -\kappa\\
\kappa   &-i \tau
\end{pmatrix}, \quad 
\mathcal M
= 
\frac12
\begin{pmatrix} 
\frac{i}{\kappa} \left(m_{31}^{\prime} + m_{21} \tau \right) & -m_{21} - i m_{31}  \\
m_{21} - i m_{31}   & 
-\frac{i}{\kappa} \left(m_{31}^{\prime} + m_{21} \tau \right) 
\end{pmatrix},
\end{align}
where $m_{21}$ and $m_{31}$ are explicitly given by
\begin{align}\label{eq:kappataum}
 m_{21} =\kappa\int^s \dot \tau ds, \quad  m_{31} = {-\dot\kappa}
 \end{align}
 Moreover, the curvature $\kappa$ and torsion $\tau$ of $\gamma$ satisfy the pair of 
 {\rm PDEs} \eqref{eq:kappatauPLR1} and \eqref{eq:kappatauPLR2}, which is the compatibility condition between \eqref{eq:LMrealfinal}.

\item[{\rm (2)}]
 Conversely, let $\kappa$ and $\tau$ be solutions to the system \eqref{eq:kappatauPLR1} and \eqref{eq:kappatauPLR2} and let $\gamma$ be the evolution determined by \eqref{eq:gammadot} and \eqref{eq:abcPLR}, where 
 the integration and derivative are taken with respect to the length element.
 Then $|\gamma'|$ does not depend on $t$ and thus $\gamma$ gives an 
 isoperimetric curve flow.
 Moreover, by choosing the arc-length parametrization $|\gamma'|=1$ of $\gamma$, the curve $\gamma$ 
 satisfies the Lund-Regge evolution $\dot \gamma^{\prime} = \gamma^{\prime} \times \dot \gamma$
 with unit-speed.
\end{enumerate}
\end{theorem}
\begin{proof}
The first statement is a consequence of Proposition \ref{prp:curve}. Let 
us consider the second statement. By the Frenet-Serre equation in 
\eqref{eq:Frenet-serre} and \eqref{eq:abcPLR}, we have 
\[ 
\langle \dot \gamma', \gamma'\rangle =  |\gamma'|^2 
\left(\frac{a'}{|\gamma'|} - b \kappa\right)=0.
\]
 Therefore, the length element $|\gamma'|$ does not depend on $t$, and 
 without loss of generality we can assume $|\gamma'| = 1$ for all $t$.
 Then the Lund-Regge condition \eqref{eq:LR} can be satisfied by 
 the choices of $a, b$ and $c$ in \eqref{eq:abcPLR} and 
 the equations in \eqref{eq:kappatauPLR1} and \eqref{eq:kappatauPLR2}.
\end{proof}

 
 Note that $\mathcal M$ in \eqref{eq:LMrealfinal} can be explicitly 
 represented by
 \begin{equation}\label{eq:LMrealfinal2}
\mathcal M
= 
\frac12
\begin{pmatrix} 
i\left(-\frac{\dot \kappa^{\prime}}{\kappa} + \tau \int^s \dot \tau ds\right) 
&  - \kappa \int^s \dot \tau ds + i \dot \kappa \\
 \kappa \int^s \dot \tau ds + i \dot \kappa & 
-i\left(-\frac{\dot \kappa^{\prime}}{\kappa} + \tau \int^s \dot \tau ds\right) 
\end{pmatrix}.
\end{equation}
Finally, we arrive at a straightforwardformulation of the Lund-Regge evolution as follows:
\begin{corollary}\label{coro:PLRevolution} 
 By choosing the diagonal gauge $\mathcal D$ taking values in 
 $\operatorname{U}(1) = \{\diag (e^{i p}, e^{-ip}) \mid p\in \R\}$ as in 
 \eqref{eq:gaugeD} and 
 introducing the complex-valued function 
\begin{equation}\label{eq:q}
q =  \kappa \exp \left( i \int^s (\tau - 1) ds\right),
\end{equation}
 the evolution of the gauged Frenet frame $F = \mathcal F \mathcal D$ with $\mathcal L$ in \eqref{eq:LMrealfinal} 
 and $\mathcal M$ in \eqref{eq:LMrealfinal2}
 can be rephrased as
\begin{equation}\label{eq:Laxpair}
\begin{split}
F^{\prime} =\;& F L,\quad
L = 
\frac{1}2
\begin{pmatrix} 
i & q\\
-\bar{q}   &-i  
\end{pmatrix},\\
\dot F =\;& F M,\quad
M = 
\frac{i} {2}
\begin{pmatrix} 
-\Re \left( \dot q^{\prime}/q\right)& 
  -\dot q  \\
  -\dot {\bar{q}}   &  \Re \left( \dot q^{\prime}/q\right)
\end{pmatrix}.
\end{split}
\end{equation}
 The compatibility condition of the above system can be computed as
$\left\{\Re \left( \dot q^{\prime}/q\right)\right\}^{\prime} = - \frac12 \left(|q|^2\right)^{\boldsymbol{\cdot}},
$ and $\Im \left(\dot q^{\prime}/q\right)=0$,
 which are equivalent to the following nonlinear PDE:
\begin{equation}\label{eq:qPLR}
 \dot q^{\prime} + \frac12 q \int^s \left(|q|^2\right)^{\boldsymbol \cdot} ds =0.
\end{equation}
\end{corollary}
\begin{proof}
 Let 
\begin{equation}\label{eq:gaugeD}
\mathcal D = \begin{pmatrix} i \exp (-\frac{i}2  \int^s (\tau - 1) ds) & 0 \\ 0 & -i\exp (\frac{i}2  
\int^s (\tau - 1) ds)  \end{pmatrix}.
\end{equation} The matrices $\mathcal L$ in \eqref{eq:LMrealfinal} and $\mathcal M$ in \eqref{eq:LMrealfinal2} are computed as $L = \mathcal D^{-1} \mathcal L \mathcal D + \mathcal D^{-1} \mathcal D^{\prime}$ and $M = \mathcal D^{-1} \mathcal M \mathcal D + \mathcal D^{-1} \dot {\mathcal D}$, respectively. The matrices $L$ and $M$ are explicitly given by
\begin{align}\label{eq:kappatauL}
        L &=   \frac{1}2
          \begin{pmatrix} 
            i &  \kappa e^{i\int^s (\tau - 1)ds}\\
            - \kappa e^{-i\int^s  (\tau - 1) ds}   &-i  
          \end{pmatrix},\\
    \label{eq:kappatauM}
        M &= \frac{i}{2}
          \begin{pmatrix}
              -\frac{\dot \kappa^{\prime}}{\kappa} + (\tau - 1)\int \dot \tau ds  & -( \dot \kappa + i \kappa \int \dot \tau ds) e^{i\int (\tau - 1) ds} \\
              -( \dot \kappa - i  \kappa \int \dot \tau ds) e^{i\int (\tau - 1) ds}& -\left(-\frac{\dot \kappa^{\prime}}{\kappa} + (\tau - 1)\int \dot \tau ds\right)
          \end{pmatrix}.
\end{align}
Next, it can be verified that $\dot q$ and $\dot q^{\prime} / q$ can be computed as follows:
\begin{gather}
    \dot q = \left( \dot \kappa + i \kappa \int^s \dot \tau ds\right) e^{i\int (\tau - 1) ds}, \label{eq:qdot}\\
    \frac{\dot q^{\prime}}{q} = \left(\frac{\dot \kappa^{\prime}}{\kappa} -(\tau - 1)\int^s \dot \tau ds+ \frac{i }{\kappa} (\kappa \dot \tau + \dot \kappa (\tau - 1) +\kappa^{\prime}\int^s \dot \tau ds) \right).\label{eq:qdotprime}
\end{gather}
We now obtain \eqref{eq:Laxpair} by substituting \eqref{eq:q}, \eqref{eq:qdot}, and \eqref{eq:qdotprime}
into each component of matrices $L$ and $M$ in \eqref{eq:kappatauL} and \eqref{eq:kappatauM}, respectively.
The compatibility conditions of the matrices $L$ and $M$ in \eqref{eq:Laxpair} are equivalent to  
\[
\left\{\Re \left( \dot q^{\prime}/q\right)\right\}^{\prime} = - \frac12 \left(|q|^2\right)^{\boldsymbol{\cdot}},\quad
\Im \left(\dot q^{\prime}/q\right)=0
\]
by a similar computation in \eqref{eq:su2generalLMcomp}.
We finally obtain \eqref{eq:qPLR} by integrating both sides of $\left\{\Re \left( \dot q^{\prime}/q\right)\right\}^{\prime} = - \frac12 \left(|q|^2\right)^{\boldsymbol{\cdot}}$ with respect to the curve parameter \textit{s} and using $\Im \left(\dot q^{\prime}/q\right)=0$.
\end{proof}

\begin{corollary}\label{coro:Laxpair}
The complex-valued function $q$ in \eqref{eq:q} can be identified with the complex-valued function $a$ in \eqref{eq:PLRaomega}, that is, $q$ gives a solution of the Pohlmeyer-Lund-Regge equation.
 Thus the gauged Frenet frame $F$ in \eqref{eq:Laxpair} is 
 the matrix-valued wave function $\Psi$ in \eqref{eq:DateLaxpair} with $\lambda =1$ 
 up to multiplication of some constant matrix by left, 
  that is, $F_0 F = \Psi$ for some $F_0 \in  \operatorname{SU}(2)$. 
 Moreover, there exists a family $\{F^{\lambda}\}_{\lambda >0}$ 
 such that $F^{\lambda}|_{\lambda =1} = F_0 F$ and $F^{\lambda}$ satisfies 
 the following system of PDEs:
 \begin{equation}\label{eq:lambdaLaxpair}
\begin{split}
(F^{\lambda})^{\prime} =\;& F^{\lambda} L^{\lambda},\quad
L^{\lambda} = 
\frac{1}2
\begin{pmatrix} 
i\lambda & q\\
-\bar{q}   &-i\lambda  
\end{pmatrix},\\
(F^{\lambda})^{\boldsymbol{\cdot}} =\;& F^{\lambda} M^{\lambda},\quad
 M^{\lambda}
= 
\frac{i} {2\lambda}
\begin{pmatrix} 
-\Re \left( \dot q^{\prime}/q\right)& 
  -\dot q  \\
  -\dot {\bar{q}}   &  \Re \left( \dot q^{\prime}/q\right)
\end{pmatrix}.
\end{split}
\end{equation}
\end{corollary}
\begin{proof}
Note that we set 
\begin{equation*}
    q = a = -\frac{\{ \exp (i \omega) \sin u \}^{\prime}}{\cos u}.
\end{equation*}
A straightforward computation using \eqref{eq:PLRaomega} shows that 
\[
- \cos u = \Re \left(\frac{\dot q^{\prime}}{q}\right),
\quad 
\exp (i \omega ) \sin u = \dot q.
\]
Therefore, we can show the claim by comparing the right-hand side of the first equation in \eqref{eq:DateLaxpair} with $L^{\lambda}$ in \eqref{eq:lambdaLaxpair} and the right-hand side of the second equation in \eqref{eq:DateLaxpair} with $M^{\lambda}$ in \eqref{eq:lambdaLaxpair}, respectively.
\end{proof}

\begin{remark}\label{rem}
\mbox{}
\begin{enumerate}
\renewcommand{\theenumi}{\rm (\arabic{enumi})}
\renewcommand{\labelenumi}{\theenumi}
\item The integrals $\int^s \left(|q|^2\right)^{\boldsymbol \cdot}  ds$ and $\int^s (\tau - 1) ds$ represent the indefinite integrals of $\left(|q|^2\right)^{\boldsymbol \cdot}$ and $\tau - 1$ with respect to the curve parameter $s$, allowing arbitrary functions of the evolution parameter $t$ as integration constants.

 \item \label{rem:sG} If the complex-valued function $q$ reduces to the real-valued function, that is, the torsion $\tau$ is identically $1$, then \eqref{eq:qPLR} becomes the {\rm sine-Gordon} equation, see \cite{NSW}$:$ let us define $u = \int^s \kappa ds$ and $f(s, t)$ by $\dot \kappa = \sin u + f$, thus $\dot u^{\prime} = \sin u + f$. Then \eqref{eq:qPLR} becomes $f^{\prime} + u^{\prime} \int^s u^{\prime} f ds =0$, and it can be written as 
 \[
\frac{d f}{d u} + \int^{u} f d u=0.
 \]
 The general solution is given by $f = A \cos u + B \sin u$ and one obtains $\dot u^{\prime} = C \sin (u + u_0)$, and the claim follows. Moreover, the $L$ and $M$ in \eqref{eq:Laxpair} can be rephrased as
 \begin{align}\label{eq:LMsinegordon}
 L = 
\frac12
\begin{pmatrix} 
i& u^{\prime}\\
-u^{\prime}  &-i 
\end{pmatrix},\quad 
 M
= \frac{i} {2}
\begin{pmatrix} 
\cos u  &  - \sin u   \\
-\sin u    & 
- \cos u 
\end{pmatrix},
\end{align}
since $\dot \kappa = \dot u^{\prime} = \sin u$ holds, which have been known the Lax pair for the sine-Gordon equation \cite{AKNS} under spectral parameter is equal to $1$. 
    \item  It is well known that 
    the nonlinear Schr\"{o}dinger equation  $i \dot q  + q^{\prime \prime} + 2 |q|^2 q =0$
     can be obtained by the evolution of $\dot \gamma = \gamma^{\prime\prime}  \times \gamma^{\prime}$, see \cite{Hashimoto}. It is an analogue that the complex nonlinear equation $\dot q^{\prime}  + \frac12 q \int^s (|q|^2)^{\boldsymbol{\cdot}} ds=0$  can be obtained by the Lund-Regge evolution 
      $\dot \gamma^{\prime} = \gamma^{\prime} \times \dot \gamma$. \label{rem:nonlinearschrodinger}
\item The Pohlmeyer-Lund-Regge equation \eqref{eq:PLR} and 
 the nonlinear PDEs \eqref{eq:kappatauPLR1}, \eqref{eq:kappatauPLR2} are equivalent. 
 Moreover, they are equivalent to the single complex nonlinear PDE \eqref{eq:qPLR}.
 We also call the PDE \eqref{eq:qPLR} as the {\rm Pohlmeyer-Lund-Regge equation}.
 This complex equation was also discussed in \cite{FukumotoMiyajima}. \label{rem:equivalency}
\end{enumerate}

\end{remark}

\subsection{The representation formula}\label{sbsc:rep}
 Recall that the Euclidean three-space $\R^3$ can be identified with $\mathfrak{su}(2)$ by 
\begin{equation}\label{eq:identification2}
\begin{pmatrix} 
p & q & r
\end{pmatrix}^{\mathsf{T}}
\in \R^3  \longleftrightarrow 
\frac{1}{2}
\begin{pmatrix} 
i r &  -p - i q \\
p- i q &  -ir 
\end{pmatrix}  
\in \mathfrak{su}(2).
\end{equation}
Here we note that the bracket $[\cdot, \cdot]$ and the trace of the matrix product in $\mathfrak{su}(2)$ can be 
translated to the cross product $\times$ and the inner product $\langle \cdot,\cdot \rangle$ in $\R^3$ as
\begin{equation}\label{eq:bracket}
a \times b = [a, b],\quad
\langle a,b \rangle =   -2\tr{(ab)}. 
\end{equation}
\begin{theorem}[The representation formula]\label{thm:Representation}
 Let $q$ be a solution of \eqref{eq:qPLR} and $F(=F^{\lambda})$ the solution of 
 the Lax pair in \eqref{eq:lambdaLaxpair}.
 Define an $\mathfrak{su}(2)$-valued map 
\begin{equation}\label{eq:Sym}
\gamma = \lambda (\partial_{\lambda}  F) F^{-1} |_{\lambda = 1},
\end{equation}
 where $\partial_{\lambda}= \frac{\partial }{\partial \lambda}$.
 Then  under the identification $\R^3 \cong \mathfrak{su}(2)$, $\gamma$ is the Lund-Regge evolution in \eqref{eq:Lund-Regge} with $|\gamma^{\prime}|=1$.
 Conversely, all Lund-Regge evolutions can be obtained by this way.
\end{theorem}
\begin{proof}
 The derivatives of $\gamma$ in \eqref{eq:Sym} can be computed as 
\[
\gamma^{\prime}= F(\lambda \partial_{\lambda} L)  F^{-1}|_{\lambda = 1},\quad
\dot \gamma= F(\lambda \partial_{\lambda} M)  F^{-1}|_{\lambda = 1} 
\]
and thus $-2\tr(\gamma^{\prime}\gamma^{\prime}) = 1$. Furthermore, 
$-2\tr(\dot \gamma \dot \gamma) = \Re(\dot q^{\prime}/q)^2 + \dot q \dot{\bar q}$
and it does not depend on the parameter $s$. Then
\begin{align*}
[\gamma^{\prime}, \dot \gamma] &= F[\lambda \partial_{\lambda} L, \lambda \partial_{\lambda} M]F^{-1}|_{\lambda = 1} 
= F[M, \lambda \partial_{\lambda} L]F^{-1}|_{\lambda = 1}  
\end{align*}
follows. On the one hand, 
\begin{align*}
 (\gamma^{\prime})^{\boldsymbol \cdot}  &= (F(\lambda \partial_{\lambda} L)  F^{-1} )^{\boldsymbol \cdot}  |_{\lambda = 1 }
 \\
 &=  F(\lambda \partial_{\lambda} \dot L + [M, \lambda \partial_{\lambda}L])F^{-1} |_{\lambda = 1} \\
 & =  F([M, \lambda \partial_{\lambda}L])F^{-1}|_{\lambda = 1} 
\end{align*}
holds. By the identification \eqref{eq:bracket}, 
$\gamma$ is the Lund-Regge evolution.
 Conversely, any Lund-Regge evolution $\gamma$ satisfies \eqref{eq:Laxpair} and 
  there exists a corresponding solution $q$ in \eqref{eq:qPLR}.
 Then it can be recovered by the formula in \eqref{eq:Sym}. This completes the proof.
\end{proof}

\begin{remark}
\mbox{}
\begin{enumerate}
\item[\rm{(1)}] The norm of $\dot \gamma$ in \eqref{eq:Sym} 
 is a function of the deformation parameter only and it is not constant in general.
\item[\rm{(2)}] 
Let 
 \begin{equation}\label{eq:Symforlambda}
\tilde \gamma = (\lambda \partial_{\lambda}  F) F^{-1} |_{\lambda >0}.
\end{equation}
 Then $\tilde \gamma$ also gives the Lund-Regge evolution with $|\tilde \gamma^{\prime}|= \lambda$, and
 the family $\{\tilde \gamma\}_{\lambda>0}$ gives a non congruent family of the Lund-Regge evolutions
 in general.
\end{enumerate}
\end{remark}
\subsection{\texorpdfstring{$N$-}-soliton curves}\label{sbsc:Nsolitoncurves}
We refer to the curves with the Lund-Regge evolution associated with $N$-soliton solutions of the Pohlmeyer-Lund-Regge equation as the \textit{$N$-soliton curves}.
We will construct $N$-soliton curves by using the wave function $\Psi(s,t,\lambda)$ in \eqref{eq:NsolitonLaxpair}.

\begin{proposition} \label{prp:Nsolitoncurve}
Let $\Psi(s,t,\lambda)$ be $2\times 2$ matrix-valued wave function in \eqref{eq:NsolitonLaxpair} and define complex-valued functions $e(s,t,\lambda),\ f(s,t,\lambda)$ and $g(s,t,\lambda)$ by 
\begin{align}
    e(s,t,\lambda) &= \exp{\left(\frac{i}{2}(\lambda s+\lambda^{-1}t)\right)}, \label{eq:e}\\
    f(s,t,\lambda) &= \Psi_{11}(s,t,\lambda)e(s,t,-\lambda)
                   = \lambda^N+\sum_{j=0}^{N-1}\nolimits\frac{d_{j+1}}{d_0}\lambda^j, \label{eq:f}\\
    g(s,t,\lambda) &=\Psi_{21}(s,t,\lambda)e(s,t,\lambda)
                   = -\sum_{j=0}^{N-1}\nolimits\frac{d_{N+j+1}}{d_0}\lambda^j, \label{eq:g}
\end{align}
where $\Psi_{ij}(s,t,\lambda)$ for $i,j \in \{1,2\}$ denotes the $(i, j)$-entry of $\Psi(s,t,\lambda)$ and $d_k$ for $k \in \{0,\dots,2N\}$ is the function given in \eqref{eq:explicitpsi}.
Then the explicit formula for the $N$-soliton curves $\gamma(s,t)=(\gamma_1(s,t),\gamma_2(s,t),\gamma_3(s,t))^{\mathsf{T}}$ 
is given as follows:
\begin{align}
    \gamma_1(s,t)&=\left. 2\Re\left[\frac{\bar{f}\partial_{\lambda}g - g\partial_{\lambda}\bar{f}}{f\bar{f}+g\bar{g}}e^2\right]\right|_{\lambda=1},\label{eq:gamma1} \\
    \gamma_2(s,t)&=\left. -2\Im\left[\frac{\bar{f}\partial_{\lambda}g - g\partial_{\lambda}\bar{f}}{f\bar{f}+g\bar{g}}e^2\right]\right|_{\lambda=1},\label{eq:gamma2} \\
    \gamma_3(s,t)&=\left.s-t+2\Im\left[\frac{\bar{f}\partial_{\lambda}f + g\partial_{\lambda}\bar{g}}{f\bar{f}+g\bar{g}}\right]\right|_{\lambda=1},\label{eq:gamma3}
\end{align}
where $e,f$ and $g$ denote $e(s,t,\lambda),f(s,t,\lambda)$ and $g(s,t,\lambda)$ in \eqref{eq:e}, \eqref{eq:f} and \eqref{eq:g}, respectively.
\end{proposition}
\begin{proof}
    The spectral parameter $\lambda$ can take values in $\C^{\times}$, but here we restrict it to positive, because $\lambda$ in Lund-Regge evolution should be the length of the curve. 
     We can normalize $\Psi(s,t,\lambda)$ in \eqref{eq:DateLaxpair} to obtain $F(s,t,\lambda)$ in $\SU$. 
     In fact, we define:
    \begin{equation}
        F(s,t,\lambda) = \frac{1}{\sqrt{\det \Psi(s,t,\lambda)}}\Psi(s,t,\lambda).
    \end{equation}
    By using the definitions of $f$, $g$ and $e$, we can represent $F(s,t,\lambda)$ as follows:
    \begin{equation}\label{eq:Fwithfg}
        F = \frac{1}{\sqrt{f\bar{f} + g\bar{g}}}
        \begin{pmatrix}
            fe & -\bar{g}e \\
            g\bar{e} & \overline{fe}
        \end{pmatrix}.
    \end{equation}
         We now define a map $\Tilde{\gamma}(s,t,\lambda)$ using $F(s,t,\lambda)$ in \eqref{eq:Fwithfg} and 
         the Sym formula in \eqref{eq:Symforlambda} as $\Tilde{\gamma} = (\lambda \partial_{\lambda}  F) F^{-1}|_{\lambda>0}$. A straightforward computation shows that the entries of $\Tilde{\gamma}(s,t,\lambda)$ can be computed as 
    \begin{align*}
        &\Tilde{\gamma}_{11}(s,t,\lambda) = \frac{i}{2}\left(\lambda s - \lambda^{-1}t + 2\lambda\frac{\Im[\bar{f}\partial_{\lambda}f + \bar{g}\partial_{\lambda}g]}{f\bar{f} + g\bar{g}}\right) \in i\R,\\
        &\Tilde{\gamma}_{21}(s,t,\lambda) = \frac{\lambda(\bar{f}\partial_{\lambda}g-g\partial_{\lambda}\bar{f})e^2}{f\bar{f} + g\bar{g}} \in \C,\\
        &\Tilde{\gamma}_{12}(s,t,\lambda) = -\overline{\Tilde{\gamma}_{21}(s,t,\lambda)},\quad \Tilde{\gamma}_{22}(s,t,\lambda) = -\Tilde{\gamma}_{11}(s,t,\lambda),
    \end{align*}
    where $f=f(s,t,\lambda),\ g=g(s,t,\lambda)$ and $e= e(s,t,\lambda)$.
    By using the identification \eqref{eq:identification2} and the representation formula \eqref{eq:Sym}, we can show that the $N$-soliton curves $\gamma(s,t) =\tilde \gamma (s, t, \lambda =1)$ are given by
    \begin{equation}
        \gamma_1(s,t) = 2\Re[\Tilde{\gamma}_{21}(s,t,1)],\ \gamma_2(s,t) = -2\Im[\Tilde{\gamma}_{21}(s,t,1)],\ \gamma_3(s,t) = 2\Im[\Tilde{\gamma}_{11}(s,t,1)].
    \end{equation}
 It can be verified that they are explicitly given in 
 \eqref{eq:gamma1}, \eqref{eq:gamma2}, and \eqref{eq:gamma3}, respectively. 
\end{proof}
Next, by using the representation formula for $N$-soliton curves and taking some parameters $\alpha_1,\dots,\alpha_N$ and $c_1,\dots,c_N$ as in \cite{Date}, we give the surface formed by 
 the $N$-soliton curve.
 \begin{example}
 We give surfaces induced by $1,2,$ and $3$-soliton solutions for PLR and sine-Gordon equations in Figure \ref{fig:1-3swptsurfaces}. Let us choose parameters for the surfaces as follows:
 \begin{enumerate}
     \item[\rm{(A)}] $(c_1,\alpha_1)=(1,e^{i\arccos{\frac{1}{4}}})$,
     \item[\rm{(B)}] $(c_1,c_2,\alpha_1,\alpha_2)=(1,1,e^{\frac{i\pi}{4}},i)$,
     \item[\rm{(C)}] $(c_1,c_2,c_3,\alpha_1,\alpha_2,\alpha_3)=(1,1,1,e^{\frac{i\pi}{6}},e^{\frac{i\pi}{3}},e^{\frac{5i\pi}{6}})$,
     \item[\rm{(D)}] $(c_1,\alpha_1)=(1,i)$,
     \item[\rm{(E)}] $(c_1,c_2,\alpha_1,\alpha_2)=(1,-1,e^{\frac{i\pi}{4}},e^{\frac{3i\pi}{4}})$,
     \item[\rm{(F)}] $(c_1,c_2,c_3,\alpha_1,\alpha_2,\alpha_3)=(1,-1,i,e^{\frac{i\pi}{6}},e^{\frac{5i\pi}{6}},i)$.
 \end{enumerate}
\begin{figure}[htbp]
    \begin{tabular}{ccc}
      \begin{minipage}[t]{0.3\hsize}
        \centering
        \includegraphics[keepaspectratio, scale=0.5]{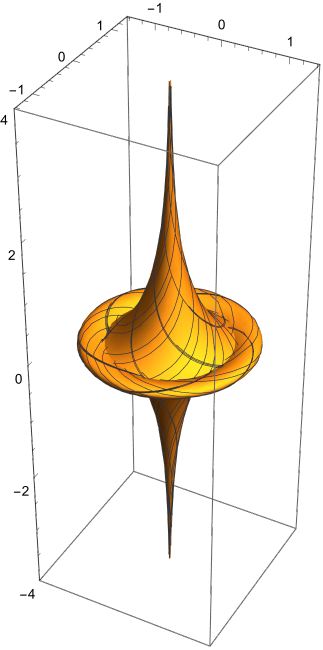}
        \subcaption{A PLR 1-soliton surface; $s,t\in[-10,10]$}
      \end{minipage} &
      \begin{minipage}[t]{0.3\hsize}
        \centering
        \includegraphics[keepaspectratio, scale=0.5]{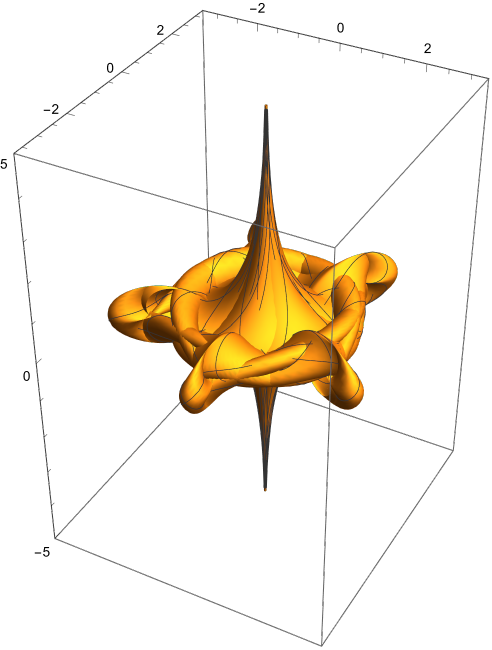}
        \subcaption{A PLR 2-soliton surface; $s,t\in[-10,10]$}
      \end{minipage} &
      \begin{minipage}[t]{0.3\hsize}
        \centering
        \includegraphics[keepaspectratio, scale=0.5]{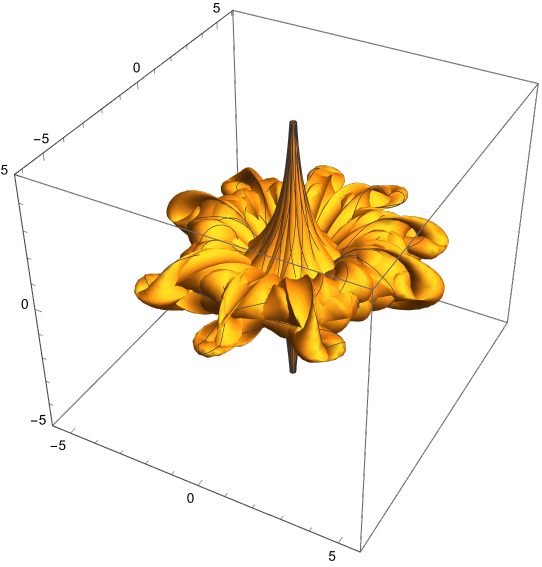}
        \subcaption{A PLR 3-soliton surface; $s,t\in[-40,40]$}
      \end{minipage} 
      \\
   
      \begin{minipage}[t]{0.3\hsize}
        \centering
        \includegraphics[keepaspectratio, scale=0.5]{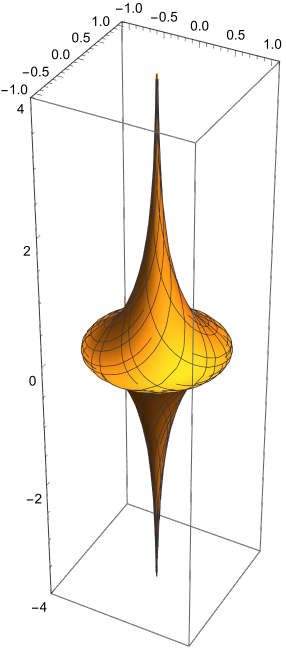}
        \subcaption{A sine-Gordon 1-soliton surface; $s,t\in[-10,10]$}
      \end{minipage} &
      \begin{minipage}[t]{0.3\hsize}
        \centering
        \includegraphics[keepaspectratio, scale=0.5]{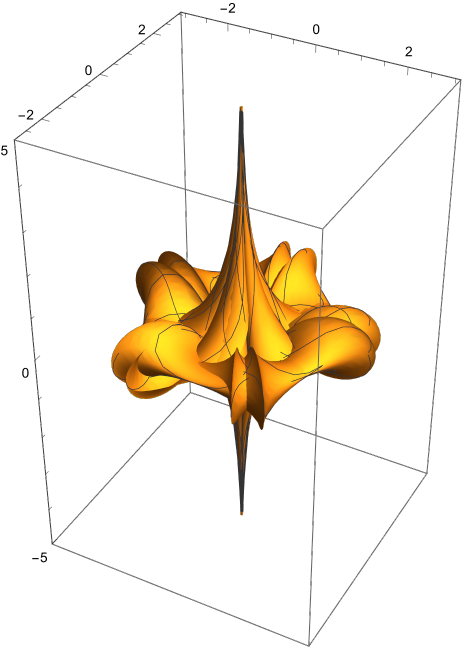}
        \subcaption{A sine-Gordon 2-soliton surface; $s,t\in[-10,10]$}
      \end{minipage} &
      \begin{minipage}[t]{0.3\hsize}
        \centering
        \includegraphics[keepaspectratio, scale=0.5]{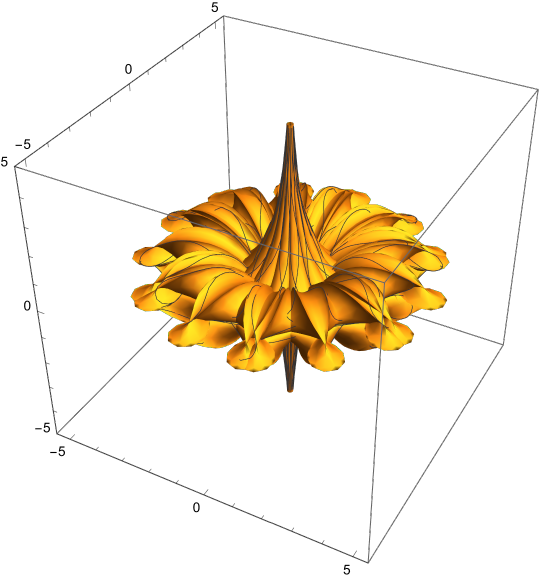}
        \subcaption{A sine-Gordon 3-soliton surface; $s,t\in[-25,25]$}
      \end{minipage} 
    \end{tabular}
     \caption{Swept surfaces formed by the curves with $1,2,$ and $3$-soliton solutions for the PLR and 
     the sine-Gordon equations.}
     \label{fig:1-3swptsurfaces}
\end{figure}
 \end{example}
\begin{example}[A $4$-soliton curve by the PLR equation]
 Let us choose 
 \[
 (c_1,c_2,c_3,c_4,\alpha_1,\alpha_2,\alpha_3,\alpha_4)=(e^{\frac{i\pi}{6}},e^{\frac{i\pi}{3}},e^{\frac{5i\pi}{6}},e^{\frac{2i\pi}{3}},e^{\frac{i\pi}{6}},e^{\frac{i\pi}{3}},e^{\frac{2i\pi}{3}},e^{\frac{5i\pi}{6}}).
 \] Then it gives a surface (a family of curves) induced by a $4$-soliton solution of the PLR equation in Figure \ref{fig:4PLRsoliton}.
\begin{figure}[htbp]
    \begin{tabular}{cc}
      \begin{minipage}[t]{0.45\hsize}
        \centering
        \includegraphics[keepaspectratio, scale=0.8]{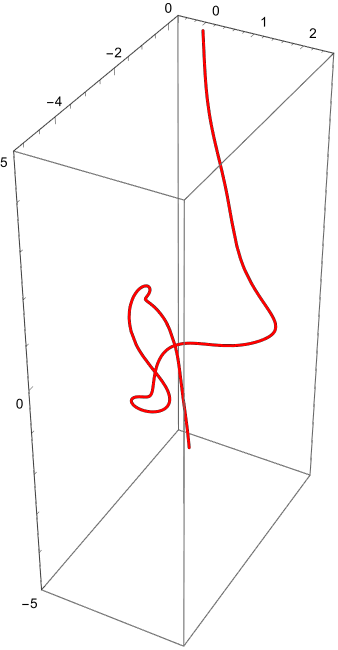}
        \subcaption{A space curve at time $t=0$; $s\in[-25,25]$}
      \end{minipage} &
      \begin{minipage}[t]{0.45\hsize}
        \centering
        \includegraphics[keepaspectratio, scale=0.8]{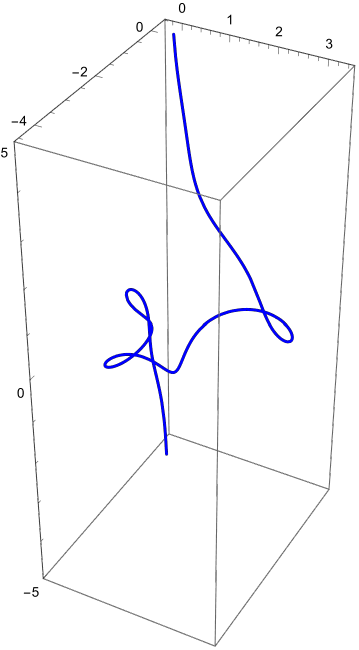}
        \subcaption{A space curve at time  $t=1$; $s\in[-25,25]$}
      \end{minipage} \\
   
      \begin{minipage}[t]{0.45\hsize}
        \centering
        \includegraphics[keepaspectratio, scale=0.8]{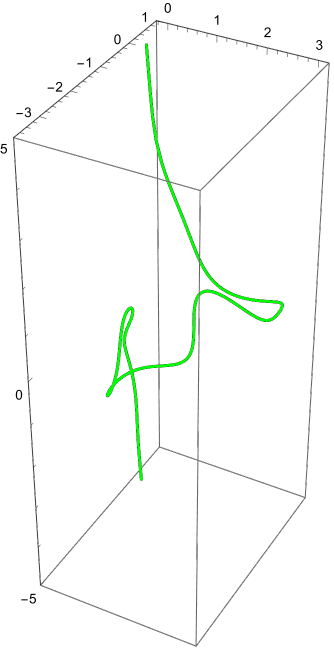}
        \subcaption{A space curve at time $t=2$; $s\in[-25,25]$}
      \end{minipage} &
      \begin{minipage}[t]{0.45\hsize}
        \centering
        \includegraphics[keepaspectratio, scale=0.8]{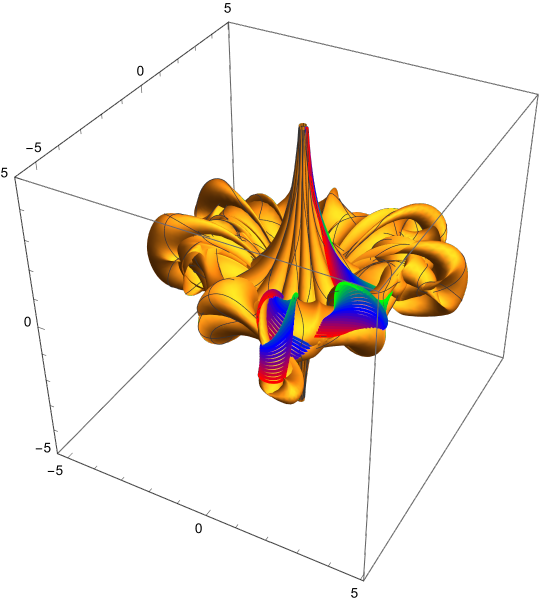}
        \subcaption{The curves of (A), (B), (C) and its swept surface; $s,t\in[-20,20]$}
      \end{minipage} 
    \end{tabular}
     \caption{A curve evolution determined from 
      a $4$-soliton solution of the PLR equation and 
      its swept surface by the curves.}
     \label{fig:4PLRsoliton}
\end{figure}
\end{example}
\begin{example}[A $4$-soliton curve by the sine-Gordon equation]
Let us choose 
\[
(c_1,c_2,c_3,c_4,\alpha_1,\alpha_2,\alpha_3,\alpha_4)=(e^{\frac{i\pi}{6}},e^{\frac{i\pi}{3}},e^{\frac{2i\pi}{3}},e^{\frac{5i\pi}{6}},e^{\frac{i\pi}{6}},e^{\frac{i\pi}{3}},e^{\frac{2i\pi}{3}},e^{\frac{5i\pi}{6}}). 
\]
Then the parameters satisfy the condition in \eqref{eq:sGcondition}, and it gives 
a surface (a family of curves) induced by a $4$-soliton solution of the sine-Gordon equation in Figure \ref{fig:4sGsoliton}.
\begin{figure}[htbp]
    \begin{tabular}{cc}
      \begin{minipage}[t]{0.45\hsize}
        \centering
        \includegraphics[keepaspectratio, scale=0.8]{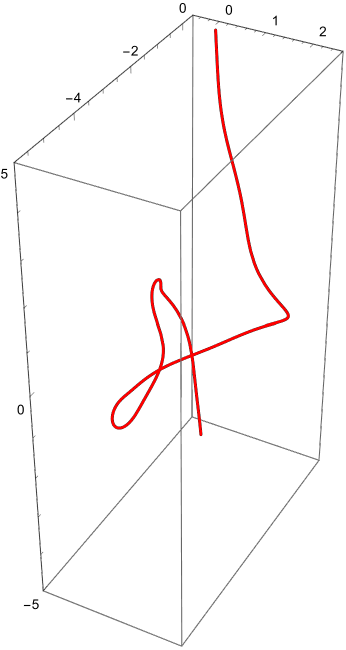}
        \subcaption{The space curve at time $t=0$; $s\in[-25,25]$}
      \end{minipage} &
      \begin{minipage}[t]{0.45\hsize}
        \centering
        \includegraphics[keepaspectratio, scale=0.8]{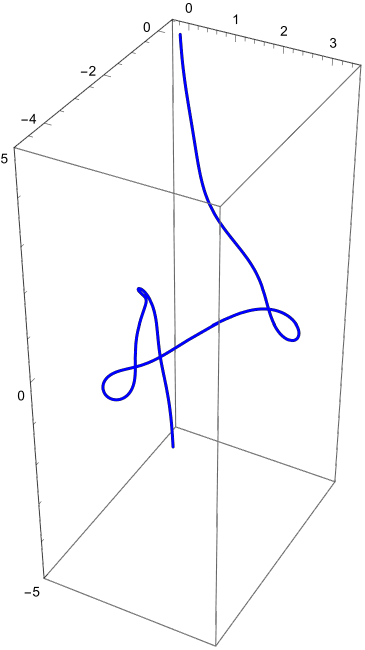}
        \subcaption{The space curve at time $t=1$; $s\in[-25,25]$}
      \end{minipage} \\
   
      \begin{minipage}[t]{0.45\hsize}
        \centering
        \includegraphics[keepaspectratio, scale=0.8]{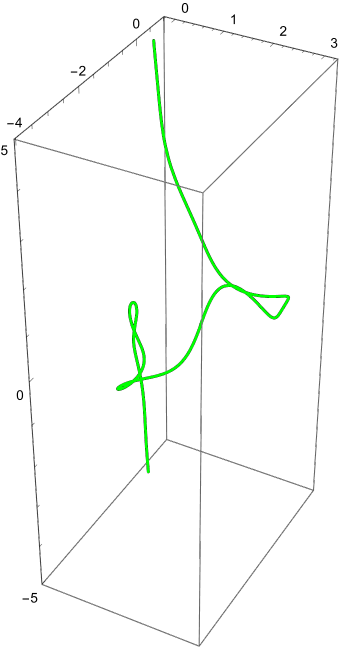}
        \subcaption{The space curve at time $t=2$; $s\in[-25,25]$}
      \end{minipage} &
      \begin{minipage}[t]{0.45\hsize}
        \centering
        \includegraphics[keepaspectratio, scale=0.8]{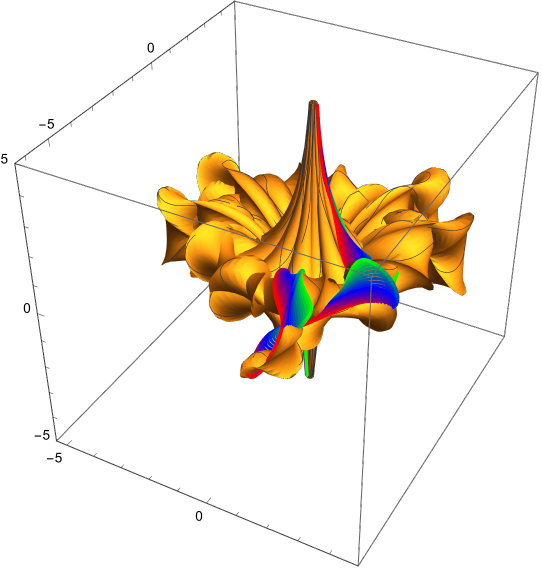}
        \subcaption{The curves of (A), (B), (C) and its swept surface; $s\in[-20,20]$}
      \end{minipage} 
    \end{tabular}
     \caption{The curve evolution determined from 
      a $4$-soliton solution of the sine-Gordon equation and 
      its swept surface by the curves.}
     \label{fig:4sGsoliton}
\end{figure}
\end{example}
\appendix
\section{The derivation of a general evolution}\label{sbsc:evo}
We now derive the time-evolution of the Frenet-frame in \eqref{eq:Fevo}. A similar computation in the proof of Proposition \ref{prp:curve} shows 
\begin{equation}\label{eq:normdot}
    |\gamma^{\prime}|^{\boldsymbol{\cdot}} = a^{\prime} - b |\gamma^{\prime}| \kappa.
\end{equation}
Differentiation of both sides of $\gamma^{\prime} = |\gamma^{\prime}|T$ with respect to the deformation parameter \textit{t} yields $\dot \gamma^{\prime} = |\gamma^{\prime}|^{\boldsymbol{\cdot}} \, T + |\gamma^{\prime}|(T)^{\boldsymbol{\cdot}}$. By substituting \eqref{eq:normdot} into this equation, we have
\begin{equation}\label{eq:Tdot}
    \dot T = \left(\frac{b^{\prime}}{|\gamma^{\prime}|} + a\kappa -c\tau \right) N + \left(\frac{c^{\prime}}{|\gamma^{\prime}|} + b\tau\right)B.
\end{equation}
Differentiation of both sides of \eqref{eq:Tdot} by the curve parameter \textit{s} shows
\begin{align}\label{eq:Tdotprime}
    \dot T^{\prime} = &-\left(\frac{b^{\prime}}{|\gamma^{\prime}|} + a\kappa -c\tau \right)|\gamma^{\prime}|\kappa T + \left(\left(\frac{b^{\prime}}{|\gamma^{\prime}|} + a\kappa -c\tau \right)^{\prime} - \left(\frac{c^{\prime}}{|\gamma^{\prime}|} + b\tau\right)|\gamma^{\prime}|\tau\right)N \\
    &+\left(\left(\frac{c^{\prime}}{|\gamma^{\prime}|} + b\tau\right)^{\prime} + \left(\frac{b^{\prime}}{|\gamma^{\prime}|} + a\kappa -c\tau \right)|\gamma^{\prime}|\tau\right)B. \notag
\end{align}
Differentiation of both sides of $T^{\prime} = |\gamma^{\prime}|\kappa N$ by \textit{t} yields $\dot T^{\prime} = (a^{\prime} - b |\gamma^{\prime}| \kappa)\kappa N + |\gamma^{\prime}|\dot \kappa N + |\gamma^{\prime}|\kappa \dot N$. By substituting \eqref{eq:Tdotprime} into this equation, we have 
\begin{align}\label{eq:Ndot}
    \dot N = &-\left(\frac{b^{\prime}}{|\gamma^{\prime}|} + a\kappa -c\tau \right)T \\
             &+\left(\left(\frac{b^{\prime}}{|\gamma^{\prime}|} + a\kappa -c\tau \right)^{\prime}\frac{1}{|\gamma^{\prime}|\kappa} - \left(\frac{c^{\prime}}{|\gamma^{\prime}|} + b\tau\right)\frac{\tau}{\kappa} - \left(\frac{a^{\prime}}{|\gamma^{\prime}|} - b \kappa\right) - \frac{\dot \kappa}{\kappa}\right)N \notag\\
             &+ \left(\left(\frac{c^{\prime}}{|\gamma^{\prime}|} + b\tau\right)^{\prime}\frac{1}{|\gamma^{\prime}|\kappa} + \left(\frac{b^{\prime}}{|\gamma^{\prime}|} + a\kappa -c\tau \right)|\gamma^{\prime}|\frac{\tau}{\kappa}\right)B. \notag
\end{align}
We note that $\langle N,N \rangle = 1$. Differentiation of both sides by \textit{t} yields $\langle \dot N,N \rangle = 0$, which implies that $\dot N$ and $N$ are orthogonal. From this and \eqref{eq:Ndot}, we have
\begin{equation*}
    \dot \kappa = \left(\frac{b^{\prime}}{|\gamma^{\prime}|} + a\kappa -c\tau \right)^{\prime}\frac{1}{|\gamma^{\prime}|} - \left(\frac{c^{\prime}}{|\gamma^{\prime}|} + b\tau\right)\tau - \left(\frac{a^{\prime}}{|\gamma^{\prime}|} - b \kappa\right)\kappa.
\end{equation*}
This equation is equivalent to \eqref{eq:compatibility2} under \eqref{eq:compatibility1}.
From the above computation, we show that the time evolution of Frenet-frame can be expressed by \eqref{eq:Fevo}.
 
\bibliographystyle{plain}

 \end{document}